\documentclass{amsproc}

\usepackage{amsmath, amsfonts,amssymb, amscd, latexsym, graphicx, psfrag, color}
\usepackage[all]{xy}

\newtheorem{dummy}{dummy}[section]
\newtheorem{lemma}[dummy]{Lemma}
\newtheorem{theorem}[dummy]{Theorem}

\newtheorem{conjecture}[dummy]{Conjecture}

\newtheorem{proposition}[dummy]{Proposition}
\theoremstyle{definition}
\newtheorem{definition}[dummy]{Definition}

\newtheorem{remark}[dummy]{Remark}

\newcommand{\bC}{\mathbb{C}}

\newcommand{\bN}{\mathbb{N}}
\newcommand{\bP}{\mathbb{P}}

\newcommand{\bR}{\mathbb{R}}
\newcommand{\bZ}{\mathbb{Z}}
\newcommand{\bT}{\mathbb{T}}

\newcommand{\cE}{\mathcal{E}}
\newcommand{\cF}{\mathcal{F}}
\newcommand{\cG}{\mathcal{G}}

\newcommand{\cL}{\mathcal{L}}

\newcommand{\cO}{\mathcal{O}}

\newcommand{\cS}{\mathcal{S}}

\newcommand{\gk}{\kappa}
\newcommand{\Fuk}{\mathrm{Fuk}}
\newcommand{\Sh}{\mathit{Sh}}

\newcommand{\Perf}{\mathcal{P}\mathrm{erf}}

\newcommand{\cpm}{\mathrm{CPM}}

\newcommand{\Tbi}{T_{\beta_{i}}}
\newcommand{\Tbj}{T_{\beta_{j}}}

\newcommand{\Ta}{T_{\alpha}}
\newcommand{\Txi}{T_{\kappa(x_i)}}

\newcommand{\To}{T_{\mathcal{O}}}

\numberwithin{equation}{section}

\begin{document}

\title{HMS for punctured tori and categorical mapping class group actions}


%    author one information
\author{Nicol\`o Sibilla}
\address{Department of Mathematics, Northwestern University, 2033 Sheridan Road, Evanston, Il 60208}

\email{sibilla@math.northwestern.edu}
\thanks{It is a pleasure to thank David Treumann and Eric Zaslow for our collaborations, and many stimulating discussions.}


\begin{abstract}
Let $X_n$ be a cycle of $n$ projective lines, and $\bT_n$ a symplectic torus with $n$ punctures. In this paper we review results of \cite{STZ} and \cite{Si}, which establish a version of homological mirror symmetry relating $X_n$ and $\bT_n$, and define on $D^b(Coh(X_n))$ an action of the pure mapping class group of $\bT_n$. 
\end{abstract}

\maketitle


\bibliographystyle{amsplain}

\section{Introduction.}
\label{intro}

As originally formulated by Kontsevich \cite{K}, Homological Mirror Symmetry (from now on, HMS) relates the derived category of coherent sheaves on a smooth projective Calabi-Yau manifold $X$, $D^b(Coh(X))$, and the Fukaya category of a compact symplectic manifold $\tilde{X}$, by stating that if $X$ and $\tilde{X}$ are mirror partners, then $D^b(Coh(X)) \cong Fuk(\tilde{X})$. Since its proposal, much work has been done towards establishing Kontsevich's conjecture in important classes of examples, see \cite{PZ, S, Sh}, and references therein.

In \cite{STZ}, joint with Treumann and Zaslow, we address mirror symmetry in dimension $1$, by proving a version of HMS which pairs singular degenerations of elliptic curves, given by cycles of projective lines, and punctured symplectic tori. This result relies on the use of a conjectural model for the Fukaya category of a punctured Riemann surface $\Sigma$, which is constructed in terms of a sheaf of dg categories, $\cpm(-)$,\footnote{$\cpm$ stands for `constructible plumbing model.' The name depends on the fact that $\cpm$ can be defined in greater generality, and supplies a dg model for the Fukaya category of \emph{plumbings} of cotangent bundles in any dimension \cite{STZ2}.} defined over the Lagrangian skeleton of $\Sigma$.\footnote{HMS for punctured spheres has been investigated also in \cite{AAEKO}.} In \cite{Si}, using the theory of spherical objects and twist functors introduced by Seidel and Thomas in \cite{SeT}, we test one of the predictions of the mirror symmetry framework developed in \cite{STZ}, by showing that the (pure) mapping class group of a punctured torus acts by equivalences on the derived category of a cycle of projective lines. 

In this paper we will review these results, by focusing on motivations and examples, and keeping the presentation of the arguments as explicit and concrete as possible. Let $X_n$ be a cycle of projective lines, with $n$ components, and $\bT_n$ a symplectic torus with $n$ punctures. We will review the construction of $\cpm$ for $\bT_n$ in Section \ref{CPM}. The HMS statement relating the $\cpm$ model for $Fuk(\bT_n)$ and $\Perf(X_n)$ will be proved in Section \ref{hms}. We will conclude by giving, in Section \ref{mcg}, a brief overview of the results in \cite{Si}, focusing on the case of $D^b(Coh(X_2))$.

\section{A model for the Fukaya category of punctured Riemann surfaces}
\label{CPM}
Starting in 2009, in various talks, Kontsevich has argued \cite{K1} that the Fukaya category of a Stein manifold should have good local-to-global properties, and therefore conjecturally could be recovered as the global sections of a suitable sheaf of dg categories (note also \cite{S1}, and \cite{N1}). This is in keeping with previous work of Nadler and Zaslow who, in \cite{NZ} and \cite{N}, establish an equivalence between the Fukaya category of exact Lagrangians in a cotangent bundle $T^*X$, and the dg category of (complexes of cohomologically) constructible sheaves over $X$, $Sh(X)$. 

Following Kontsevich's insight, in \cite{STZ} we equip the Lagrangian skeleton of a punctured Riemann surface $\Sigma$ with a sheaf of dg categories,\footnote{$Sh(X)$ is the dg enhancement of the derived category of constructible shaves of $\bC$-vector spaces over $X$. From now on, we will refer to objects in $Sh(X)$ simply as `constructible sheaves.' See \cite{KS} for a comprehensive introduction to the subject.} called $\cpm(-)$, such that its local behavior is dictated by Nadler and Zaslow's work on cotangent bundles, while its global sections are conjecturally quasi-equivalent to the Fukaya category of exact Lagrangians in $\Sigma$, $Fuk(\Sigma)$. Before discussing the construction of $\cpm(-)$ in Section \ref{cpm}, we collect in Section \ref{microlocal} below the necessary background on sheaf theory.

\subsection{Microlocal sheaf theory in dimension $1$}
\label{microlocal}
In \cite{KS}, Kashiwara and Schapira explain how to attach to a constructible sheaf $\cF \in Sh(X)$ a conical (i.e. invariant under fiberwise dilation) Lagrangian subset of $T^*X$, called \emph{singular support}, and denoted $SS(\cF)$. Informally, $SS(\cF)$ is an invariant encoding the co-directions along which $\cF$ does not `propagate.' Rather than giving the exact definition, for which we refer the reader to Section 5.1 of \cite{KS}, we will describe in Lemma \ref{lambda} how the singular support works in the cases which will be relevant for us. If $\Lambda \hookrightarrow T^*X$ is a conical Lagrangian subset, denote $Sh(X, \Lambda) \hookrightarrow Sh(X)$ the full subcategory of constructible sheaves $\cF \in \Sh(X)$ such that $SS(\cF) \subset \Lambda$.  

\begin{proposition} 
\label{sheaf}
Let $X$ be a $1$-dimensional manifold, let $\Lambda \hookrightarrow T^*X$ be a conical Lagrangian subset, and denote $\pi:T^*X \rightarrow X$ the natural projection. The assigment sending a conical open subset $V \subset \Lambda$, to the dg category $Sh(\pi(V), V)$, 
can be extended to a sheaf of dg categories, denoted $MSh(-)$, over $\Lambda$ equipped with its natural topology.
\end{proposition}
\begin{proof}
The proof of the statement is discussed in Section 3.1 of \cite{STZ}, and depends on the microlocal theory of sheaves developed in \cite{KS}. In fact, a similar statement holds in all dimensions \cite{STZ2}.
\end{proof}

\begin{lemma}
\label{lambda}
Let $X$ be a $1$-dimensional manifold, and $P = \{p_1, p_2 \ldots, p_n\}$ a finite collection of points in $X$, then $\Lambda_P = X \cup (T^*_{p_1}X \cup \ldots \cup T^*_{p_n}X)$ is a conical Lagrangian, and $Sh(X, \Lambda_P) \subset Sh(X)$ coincides with the full subcategory of constructible sheaves which are locally constant on $P$, and on $X \setminus P$. 
\end{lemma}

$Sh(X, \Lambda_P)$ admits a very simple combinatorial description in terms of quiver representations. Call $\cS$ the partition of $X$ given by the points in $P$, and by the connected components of $X \setminus P$. Denote $Q_\cS$ the quiver whose vertices are the elements of $\cS$, and with an arrow joining $S, S' \in \mathcal{S}$, with that orientation, if and only if $S$ is a point, and $S'$ is a sub-interval such that $S \in \overline{S'}$. For example, if $P$ has cardinality $1$, then $Q_\cS$ is equal to $\bullet \leftarrow \bullet \rightarrow \bullet$ if $X = \bR$, and to $\bullet \rightrightarrows \bullet$ if $X = S^1$.

\begin{lemma}
\label{quiver}
$Sh(X, \Lambda_P) \cong Rep(Q_\cS)$.\footnote{$Rep(Q_\cS)$ denotes the dg derived category of representations. In the proof below and everywhere in the paper, all functors, such as the stalk functor, are implictly assumed to be derived.}
\end{lemma}
\begin{proof}
This statement is part of a larger framework, due to MacPherson, which describes sheaves that are constructible with respect to a given stratification in terms of \emph{exit paths}. If $\cF \in Sh(X, \Lambda)$, by taking its stalks over $p_1, \ldots p_n$, and over points lying on the different components of $X \setminus P$, we obtain a complex of vector spaces for each vertex of $Q_\cS$. Further, the restrictions maps of $\cF$ yield linear maps corresponding to the arrows of $Q_\cS$. This prescription maps $\cF$ to a representation of $Q_\cS$ in a functorial way, and defines the equivalence.
\end{proof}

\begin{remark}
\label{wheels}
Under the assumptions of Lemma \ref{quiver}, the sheaf $MSh(-)$ can be described explicitly. Assume, for concreteness, that $X = S^1$, and $P = \{p \}$. Also, fix an orientation on $T^*S^1$, and note that this allows us to decompose $T^*_pS^1$ as the union of $0$, and two rays, $R^+$ and $R^-$, respectively up-ward and down-ward pointing. Below, we describe the sections of $MSh(-)$ on \emph{contractible} open subsets $U \subset \Lambda$, and the assignment defining, on the objects, the restriction functors 
$$
Res_U: MSh(\Lambda) = Sh(X, \Lambda) \cong Rep(\bullet \rightrightarrows \bullet) \rightarrow MSh(U),
$$ 
the definition on morphisms will be obvious. This is enough to reconstruct $MSh(-)$.
Let $M=\xymatrix{V_1 \ar@<0.5ex>[r]^{f} \ar@<-0.5ex>[r]_{g}& V_2}$ be an object in $Rep(\bullet \rightrightarrows \bullet)$, then 
\begin{itemize}
\item if $U \subset S^1$, $MSh(U) \cong \bC-mod$, and $Res_U(M) = V_2$,
\item if $U \subset R^+$, $MSh(U) \cong \bC-mod$, and $Res_U(M) = Cone(f)$,
\item if $U \subset R^-$, $Msh(U) \cong \bC-mod$, and $Res_U(M) = Cone(g)$,
\item if $p \in U$, $MSh(U) \cong Rep(\bullet \leftarrow \bullet \rightarrow \bullet)$, and $Res_U(M)= V_2 \stackrel{f}{\leftarrow} V_1 \stackrel{g}{\rightarrow} V_2$.
\end{itemize}
\end{remark}
 
\subsection{The construction of $\cpm$ for $\bT_n$}
\label{cpm} 
A \emph{ribbon graph} is a graph equipped with a cyclic ordering on the set of half-edges incident to each vertex. Recall that ribbon graphs label cells in the moduli space of punctured Riemann surfaces (see e.g. \cite{P}). Further, if the Riemann surface $\Sigma$ lies in the cell labelled by $\Gamma_{\Sigma}$, there is an embedding $\Gamma_{\Sigma} \hookrightarrow \Sigma$, and a nicely behaved retraction of $\Sigma$ onto $\Gamma_{\Sigma}$. In the language of Stein geometry, $\Gamma_{\Sigma}$ is the \emph{skeleton} of $\Sigma$. 

Bracketing issues of valency, given a pair formed by a punctured Riemann surface and its skeleton, $\Gamma_{\Sigma} \hookrightarrow \Sigma$, we can consider an open covering $\{ U_i \}_{i \in I}$ of $\Gamma_\Sigma$ with the property that, for all $i \in I$, there is a symplectomorphism $\phi_i: U_i \rightarrow T^*M_i$, where $M_i$ is a $1$-dimensional manifold, and $\Lambda_i := \phi_i(U_i \cap \Gamma_\Sigma) \hookrightarrow T^*M_i$ is a conical Lagrangian subset. In conformity with Kontsevich's ansatz, we should be able to recover $\Fuk(\Sigma)$ by first applying Nadler and Zaslow's theory \cite{NZ, N} to compute the Fukaya category of each member of the covering family (this yields $Sh(M_i, \Lambda_i)$), and then exploiting the fact that the Fukaya category behaves like a sheaf over $\Gamma_{\Sigma}$. 

Roughly speaking, the theory developed in \cite{STZ} formalizes this heuristics by constructing a sheaf of dg categories $\cpm(-)$ over $\Gamma_\Sigma$, which is characterized by the property that, when restricted to $U_i \cap  \Gamma_\Sigma \cong \Lambda_i$, it coincides with the sheaf $MSh(-)$ over $\Lambda_i$ introduced in Proposition \ref{sheaf}. Since $\{U_i\}_{i \in I}$ covers $\Gamma_\Sigma$, this prescription is actually sufficient to compute sections and restriction functors for $\cpm(-)$ on arbitrary open subsets $U \subset \Gamma_\Sigma$, and therefore determines $\cpm(-)$ uniquely.\footnote{This informal account disregards various technical aspects of the theory, for which see \cite{STZ}.} 

\begin{conjecture}
\label{conj}
$\cpm(\Gamma_\Sigma)$ is quasi-equivalent to the Fukaya category of compact exact Lagrangians in $\Sigma$, $Fuk(\Sigma)$
\end{conjecture}

For the purposes of HMS, it is important to understand how this works for a symplectic torus with $n$ punctures, $\bT_n$. Let $\Lambda_i$, $i \in I = \{1, \ldots, n\}$, be a collection of $n$ copies of the conical Lagrangian $\Lambda_p := S^1 \cup T^*_pS^1 \hookrightarrow T^*S^1$. Note that, if we fix an orientation on $T^*S^1$, $\Lambda_p$ acquires a natural structure of ribbon graph. For each $i \in I$ there are open embeddings $j^+_i: \bR_{>0} \cong R^+_i \hookrightarrow \Lambda_i$, and  $j^-_i: \bR_{>0} \cong R^-_i \hookrightarrow \Lambda_i$, where $R^+_i$ and $R^-_i$ are defined as in Remark \ref{wheels}. Denote $\Gamma_n$ the ribbon graph constructed as the push-out of the diagram
$$
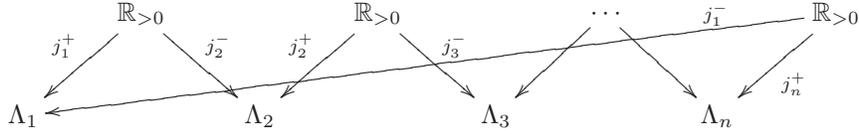
\xymatrix{ 
& \bR_{>0} \ar[ld]_{j^+_1} \ar[rd]^{j^-_2} & & \bR_{>0} \ar[ld]_{j^+_2} \ar[rd]^{j^-_3} & &\cdots \ar[dl] \ar[rd] & _{j^-_1} & \bR_{>0} \ar[dl]^{j^+_n} \ar[dlllllll]  \\
\Lambda_1 & & \Lambda_2 & & \Lambda_3 &  & \Lambda_n 
}
$$ 

Provided that $\bT_n$ is equipped with an appropriate complex structure, $\Gamma_n$ is isomorphic to the skeleton of $\bT_n$.\footnote{Note that, although a different choice of complex structure on $\bT_n$ could alter the geometry of the skeleton, this would not affect, up to quasi-equivalence, the global sections of $\cpm(-)$.} Also, the $\Lambda_i$-s supply an open covering for $\Gamma_n$. As $\cpm(-)$ defines a sheaf over $\Gamma_n$, its global sections $\cpm(\Gamma_n)$ can be calculated in the usual way, by taking the (\emph{homotopy}) equalizer of the \v{C}ech diagram (\ref{eq cpm}) below, and thus, informally, picking out local sections on the $\Lambda_i$-s which agree on the overlaps,
\begin{equation}
\label{eq cpm}
\xymatrix{
\  \underset{i \in I}{\prod}(\cpm(\Lambda_i) = Sh(S^1, \Lambda_p)) \ar@<1.5ex>[r]^{Res_+} \ar@<0.0ex>[r]_{Res_-}& \underset{j \in I}{\prod}(\cpm(\bR_{>0}) \cong \bC-mod).}
\end{equation}
Note that the functors $Res_+$ and $Res_-$ can be explicitly computed, since they are products of restriction functors for $Sh(S^1, \Lambda_p) \cong Rep(\bullet \rightrightarrows \bullet)$ which were described in Remark \ref{wheels}. We will conclude this section by giving a concrete recipe for constructing homotopy equalizers in $dgCat$.
\begin{lemma}
\label{homotopy eq}
Let $\xymatrix{\ \mathcal{C} \ar@<1ex>[r]^{F} \ar@<-0.5ex>[r]_{G}& \mathcal{D}}$ be a diagram in $dgCat$, and denote $\cE$ the dg category having,
\begin{itemize}
\item as objects, pairs $(C, u)$, where $C \in \mathcal{C}$, and $u$ is a degree zero, closed morphism 
$u:F(C) \cong G(C)$, which becomes invertible in the homotopy category,

\item as morphisms, pairs $(f, H) \in hom^k(C, C')\oplus hom^{k-1}(F(C), G(C'))$, with differential given by $d(f, H) = (df, dH - (u'F(f) - G(f) u))$. The composition is obvious.
\end{itemize}
Then $\cE$, endowed with the natural forgetful functor $\cE \rightarrow \mathcal{C}$, is a homotopy equalizer for $F$ and $G$.
\end{lemma}
\begin{proof}
Lemma \ref{homotopy eq} depends on the availabilty of an explicit construction of the \emph{path object} $P(\mathcal{D})$ for $\mathcal{D}$, which can be found in Lemma 4.1 of \cite{Tab1}. This allows us to compute the homotopy equalizer in the usual way, by taking appropriate fibrant replacements. We leave the details to the reader.
\end{proof}

\section{$\cpm$ and mirror symmetry for degenerate elliptic curves}
\label{hms}
Let $X_n$ be a cycle of $n$ projective lines. That is, $X_n$ is a connected reduced curve with $n$ nodal singularities, such that its normalization $\tilde{X_n} \stackrel{p}{\rightarrow} X$ is a disjoint union of $n$ projective lines $D_1, \dots, D_n$, with the property that the pre-image along $\pi$ of the singular set interesects each $D_i$ in exactly two points. Theorem \ref{normalization} below yields a description of $\Perf(X_n)$ as a suitable homotopy equalizer of dg categories. This is one of the key steps in the proof of HMS for $X_n$, which will be presented in Section~\ref{sec hms}.  
\subsection{Perfect complexes over a nodal curve}
It will be covenient to make use of the following general result, according to which $\Perf(-)$ defines a sheaf of dg categories for the Zariski topology.
\begin{theorem}[see \cite{Toen}, Proposition 11]
\label{gluing}
Let $X = U \cup V$, where U and V are two Zariski open subschemes. 
Then the following square:
$$
\xymatrix{
\ \Perf(X)    \ar[r] \ar[d]      & \Perf(U)  \ar[d] \\
\ \Perf(V) \ar[r]                & \Perf(U \cap V)
}
$$
is a fiber product of dg categories.
\end{theorem}

\begin{lemma}
\label{bundles}
Let $X$ be a nodal curve, with normalization $\tilde{X} \stackrel{\pi}{\rightarrow} X$, then for every  $\mathcal{\tilde{F}}$ vector bundle on $\tilde{X}$ and isomorphism $u: \sigma^*(\mathcal{\tilde{F}}) \rightarrow \tau^*(\mathcal{\tilde{F}})$, the assignment:
$$
U \subset^{open} X \mapsto \{ s \in \mathcal{\tilde{F}}( \pi^{-1}(U)) | u(\sigma^*(s)) = \tau^*(s) \},
$$
defines a vector bundle $\tilde{\cF}^u$ on $X$ such that $\pi^*(\cF^u) \cong \mathcal{\tilde{F}}$. Conversely, if $\cF$ is a vector bundle on $X$ such that $\pi^*\cF \cong \tilde{\cF}$, then $\cF \cong \tilde{\cF}^u$ for some isomorphism $u: \sigma^*(\mathcal{\tilde{F}}) \rightarrow \tau^*(\mathcal{\tilde{F}})$.
\end{lemma}
\begin{proof} 
See Proposition 4.4 in \cite{L}.
\end{proof}

Recall that a nodal curve is a curve having only double points as singularities. 
\begin{theorem}
\label{normalization}
Let $X$ be a nodal curve, with singular set $Z$, and normalization $\pi: \tilde{X} \rightarrow X$. Let $\sigma, \tau: Z \rightarrow X$ be two non-overlapping sections of $\pi^{-1}(Z) \rightarrow Z$, then the diagram 
$$
\xymatrix{
\ \Perf(X) \ar[r]^{\pi^{*}}   & D^b(Coh(\tilde{X})) \ar@<1ex>[r]^{\sigma^*} \ar@<-0.5ex>[r]_{\tau^*}& D^b(Coh(Z))}
$$
is an equalizer of dg categories.
\end{theorem}
\begin{proof}
Since limits commute with limits, it is sufficient, after Theorem \ref{gluing}, to prove the claim for affine $X$, so we will rectrict to this case. Let $E$ be the equalizer of the diagram
$$
\xymatrix{D^b(Coh(\tilde{X})) \ar@<1ex>[r]^{\sigma^*} \ar@<-0.5ex>[r]_{\tau^*}& D^b(Coh(Z))}
$$ 
constructed according to the prescriptions of Lemma \ref{homotopy eq}. Recall that the objects of $E$ are pairs $(\tilde{\cF}, u)$, where $\tilde{\cF}$ is an object of $D^b(Coh(\tilde{X}))$, and $u$ is a degree zero, closed morphism 
$u:\sigma^*\tilde{\cF} \cong \tau^*\tilde{\cF}$, which becomes invertible in the homotopy category. The morphisms of $E$ are pairs $(f, H) \in hom^k(\tilde{\cF}, \tilde{\cG})\oplus hom^{k-1}(\sigma^*\pi^*\tilde{\cF}, \tau^*\pi^*\tilde{\cG})$, and the differential is given by $d(f, H) = (df, dH - (u'\sigma^*(f) - \tau^*(f) u))$.

Fix a natural equivalence $\alpha: \sigma^* \pi^* \cong \tau^* \pi^*$. As $\Perf(X)$ is generated by line bundles, and $E$ is generated by objects of the form $(\tilde{\cF}, u)$ with $\tilde{\cF}$ a line bundle on $\tilde{X}$, it is sufficient to define a (quasi-)equivalence $\psi$ between these two linear sub-categories. Define $\psi$ as follows,
\begin{itemize}
\item if $\cF$ is a line bundle on $X$, then $\psi(\cF) = (\pi^* \cF, \sigma^*\pi^* \cF \stackrel{\alpha}{\rightarrow} \tau^*\pi^* \cF)$,
\item if $\cF, \cG$ are line bundles on $X$, and $f \in hom^k(\cF, \cG)$, then $\psi(f) = (\pi^*f, 0)$.
\end{itemize}
Consider a line bundle $\tilde{\cF}$ over $\tilde{X}$. It follows from Lemma \ref{bundles} that the set of isomorphism classes of line bundles $\cF$ on $X$ such that $\pi^*\cF \cong \tilde{\cF}$ carries a transitive action by $(\bC^*)^{|Z|}$ (given by pointwise rescaling the `compatibility' isomorphisms $u$, see Lemma \ref{bundles}). Further, the same is true for the set of isomorphism classes of objects of $(\tilde{\cG}, v) \in E$, such that $(\tilde{\cG}, v) \cong (\tilde{\cF} , u)$ for some $u \in hom^0(\sigma^*\tilde{\cF}, \tau^*\tilde{\cF})$. Essential surjectivity follows from the fact that $\psi$ defines a $(\bC^*)^{|Z|}$-equivariant map between these two sets of isomorphism classes.

We shall prove next that $\psi$ is quasi-fully faithful, i.e. that the map between hom-complexes defined by $\psi$ induces an isomorphism in the homotopy category. Denote $HoE$ the homotopy category of $E$. It is sufficient to show that for all line bundles $\cF$ on $X$, and for all $i \in \bN$,
$$
\psi: Hom_X^i(\cO_X, \cF) (= H^i_X(\cF)) \stackrel{\cong}{\rightarrow} Hom^i_{HoE}(\psi(\cO_X), \psi(\cF)).
$$ 
Note that, as $X$ and $\tilde{X}$ are affine, cohomology vanishes in positive degree. It follows that $Hom^i_{HoE}(\psi(\cO_X), \psi(\cF)) = 0$ for all $i > 0$. \footnote{Note that $Hom^1_{HoE}(\psi(\cO_X), \psi(\cF))$ vanishes, since it is isomorphic to the quotient of $Hom^0_Z(\sigma^*\cO_{\tilde{X}}, \tau^*\pi^*\cF) \cong \bC$ by the image of the differential, which is easily seen to be surjective.}. Further, in degree-zero, the hom-space fits in the following short exact sequence
$$
0 \rightarrow Hom^0_{HoE}(\psi(\cO_X), \psi(\cF)) \rightarrow Hom^0_{\tilde{X}}(\cO_{\tilde{X}}, \pi^*\cF) \rightarrow Hom^0_{Z}(\sigma^*\cO_{\tilde{X}}, \tau^*\pi^*\cF) \rightarrow 0.
$$
Thus, proving fully faithfulness boils down to showing exactness of 
\begin{equation}
\label{eq:1}
0 \rightarrow H^0_X(\cF) \stackrel{\pi^*}{\rightarrow} H^0_{\tilde{X}}(\pi^*\cF) \rightarrow H^0_{Z}(\tau^*\pi^*\cF)  \rightarrow 0.
\end{equation}
Now, (\ref{eq:1}) is obtained by taking global sections of the sequence 
$$
0   \rightarrow \mathcal{F}  \rightarrow \pi_* (\pi^* \mathcal{F})  \rightarrow \pi_*\tau_*\tau^*(\pi^*\mathcal{F}) \rightarrow    0,
$$
which is exact (see the proof of Proposition 4.4 of \cite{L}). Since $X$ is affine, taking global section is an exact operation, and this concludes the proof of Theorem \ref{normalization}. 
\end{proof}

\subsection{HMS for nodal elliptic curves}
\label{sec hms}
In this section we will prove that the category of perfect complexes over $X_n$ is quasi-equivalent to $\cpm(\Gamma_n)$. Granting Conjecture \ref{conj}, this result confirms well known mirror symmetry heuristics, which suggest that the mirror of $X_n$ should be a symplectic torus with $n$ punctures, $\bT_n$.\footnote{Kontsevich announced related results in \cite{K1}. HMS for the nodal $\bP^1$ is also treated in \cite{LP}.}
\begin{theorem}[\cite{Be}]
\label{beilinson}
There is an equivalence $\beta: D^b(Coh(\bP^1)) \stackrel{\cong}{\rightarrow} Rep(\bullet \rightrightarrows \bullet)$.
\end{theorem}
\begin{proof} For the proof, see Beilinson's famous paper \cite{Be}, which provides analogous descriptions of $D^b(Con(\bP^n))$ for any $n$. The functor $\beta$ can be defined as follows. Fix a basis $x_0, x_1$ for $H^0(\cO(1))$, and set
$$
\cF \in D^b(Coh(\bP^1)) \longmapsto \beta(\cF) = \xymatrix{R\Gamma(\cF\otimes\cO(-1)) \ar@<0.7ex>[r]^(0.6){\cdot x_0} \ar@<-0.5ex>[r]_(0.6){\cdot x_1}& R\Gamma(\cF)} \in Rep(\bullet \rightrightarrows \bullet),
$$
with the obvious definition on morphisms.
\end{proof}
In view of results of Nadler and Zaslow \cite{NZ, N} discussed in Section \ref{cpm}, Proposition \ref{beilinson}, combined with the equivalence $Rep(\bullet \rightrightarrows \bullet) \cong Sh(S^1, \Lambda_p)$, yields a homological mirror symmetry statement pairing $D^b(Coh(\bP^1))$, and a suitable Fukaya category of exact Lagrangians in $T^*S^1$. This was explained as an instance of T-duality by Fang \cite{F}, and fits in the framework of the \emph{coherent-constructible correspondence} developed by Fang, Liu, Treumann and Zaslow (see \cite{FLTZ}), which is one of the starting points for the project of \cite{STZ}.\footnote{The significance for mirror symmetry of the equivalence $D^b(Coh(\bP^1)) \cong Sh(S^1, \Lambda_p)$ was first advocated by Bondal \cite{B}, in the context of HMS for weighted projective spaces.}

\begin{theorem}[\cite{STZ}]
\label{thm hms}
Let $X_n$ be a cycle of $n$ projective lines. There is a quasi-equivalence $\phi: \Perf(X_n) \cong \cpm(\Gamma_n)$.\footnote{The result proved in \cite{STZ} is actually more general, and extends to appropriate \emph{stacky} degenerations of elliptic curves.}
\end{theorem} 
\begin{proof}
Let $Z \hookrightarrow X_n$ be the singular set. Pick two non-overlapping sections $\sigma$, and $\tau$ of $p^{-1}(Z) \rightarrow Z$, as in Theorem \ref{normalization}, and choose an identification $D^b(Coh(Z)) \cong \prod_{j=1}^{j=n} \bC-mod$. 
The proof is encoded in the following diagram,
$$
\xymatrix{
\ \overset{}{\Perf(X_n)} \ar[r] \ar@{-->}[d]^{ \phi} &  \prod_{i=1}^{i=n} D^b(Coh(\bP^1)) \ar[d]^{\beta} \ar@<1.0ex>[r]^{\sigma^*} \ar@<-0.5ex>[r]_{\tau^*} & \prod_{j=1}^{j=n} \bC-mod \ar[d]^{\rho} \\
\ \cpm({\Gamma_n}) \ar[r] & \prod_{i=1}^{i=n} Rep(\bullet \rightrightarrows \bullet) \ar@<1.0ex>[r]^{Res^+} \ar@<-0.5ex>[r]_{Res^-} &  \prod_{j=1}^{j=n} \bC-mod.
}
$$
In fact, we can choose $\rho$ in such a way that $\rho\circ\sigma^* \cong Res^+\circ\beta$, and $\rho\circ\tau^* \cong Res^-\circ\beta$, where, abusing notation, we are denoting $\times_{i=1}^{i=n}\beta$ simply by $\beta$. This implies that the equalizer of $\sigma^*, \tau^*$ is quasi-equivalent to the equalizer of $Res^+, Res^-$ (see diagram \ref{eq cpm}), and yields $\phi: \Perf(X_n)\stackrel{\cong}{\rightarrow} \cpm(\Gamma_n)$.
\end{proof}

\section{A mapping class group action on $D^b(Coh(X_n))$}
\label{mcg} 
Since the group of symplectic automorphisms of $\bT_n$ acts by auto-equivalences on $Fuk(\bT_n)$, HMS predicts the existence of a mirror action on $D^b(Coh(X_n))$.\footnote{Note in fact that, although the HMS statement of Section \ref{sec hms} involves $\Perf(X_n)$, it is possible to show that $D^b(Coh(X_n))$ and $\Perf(X_n)$ have the same group of auto-equivalences.} This is the content of the main theorem of \cite{Si}, which we state below. 
\begin{theorem}[\cite{Si}]
\label{thrm:action}
Let $\mathrm{PM}(\bT_n)$ be the pure mapping class group of $\bT_n$, then, up to shift, there is an action of $\mathrm{PM}(\bT_n)$ over $D^b(Coh(X_n))$.
\end{theorem}
Recall that the mapping class group of an oriented surface $\Sigma$, $\mathrm{MCG}(\Sigma)$, is the group of symplectic automorphisms of $\Sigma$, up to isotopy. The pure mapping class group is the subgroup $\mathrm{PM}(\Sigma) \hookrightarrow \mathrm{MCG}(\Sigma)$ generated by Dehn twists (see \cite{FM}). Theorem \ref{thrm:action} generalizes previous work of Seidel and Thomas \cite{SeT}, and Burban and Kreussler \cite{BK}, who established, respectively, the existence of an $SL(2, \bZ)$-action over $D^b(Coh(X))$, where $X$ is a smooth elliptic curve, and over $D^b(Coh(X_1))$.\footnote{Recall that the mapping class groups of the torus and of the once punctured torus are both isomorphic to $SL(2, \bZ)$.} The proof of Theorem \ref{thrm:action} depends on the availability of an explicit presentation of $\mathrm{PM}(\bT_n)$, worked out in the first Section of \cite{Si}, and on Seidel and Thomas' theory of \emph{spherical objects} and \emph{twist functors} \cite{SeT}, which is designed precisely to test this aspect of mirror symmetry. 

In the rest of this section, we shall briefly describe the proof of Theorem \ref{thrm:action} for $n=2$, leaving out most details, and referring the reader to \cite{Si} for the general case. A presentation of $\mathrm{PM}(\bT_2)$ can be found in \cite{PS}, and is reproduced below.
\begin{proposition}
\label{two}
The pure mapping class group $\mathrm{PM}(\bT_2)$ is generated by $\Ta$, $T_{\beta_1}$ and $T_{\beta_2}$,\footnote{The generators $\Ta, \Tbi$ are given by isotopy classes of Dehn twists along simple closed curves $\alpha, \beta_i \hookrightarrow \bT_2$. Explicit representatives can be described as follows. Identify the torus $\bT$ with $[0,1] \times [0,1] \diagup_\sim$, and set $\bT_2 = \bT \setminus \{ p_1 = (0,0), p_2 = (\frac{1}{2}, 0)\}$, then $\alpha = [0,1] \times \{\frac{1}{2}\}$, and $\beta_i = \{\frac{i}{3}\} \times [0,1], i=1,2$.} with relations
\begin{itemize}
\item (Braid relations) $\Tbi\Tbj = \Tbj\Tbi$, $\Tbi \Ta \Tbi = \Ta \Tbi \Ta$,
\item ($G$-relation) $(T_{\beta_1} \Ta T_{\beta_2})^4 = 1$.
\end{itemize}
\end{proposition}
Following the discussion in Section 1 of \cite{SeT}, the group acting on $D^bCoh(X_2)$ is going to be a suitable central extension of $\mathrm{PM}(\bT_2)$, whose elements should be viewed as \emph{graded} symplectic automorphisms of the mirror of $X_2$, i.e. $\bT_2$.
\begin{definition}
Define $\widetilde{\mathrm{PM}}(\bT_2)$ as the $\bZ$-central extension of $\mathrm{PM}(\bT_2)$, generated by $\Ta$, $\Tbi, i=1,2$, and a central element $t$ subject to the following relations
\begin{itemize}
\item (Braid relations), as in Proposition \ref{two}
\item ($\tilde{G}$-relation) $(T_{\beta_1} \Ta T_{\beta_2})^4 = t^2$.
\end{itemize}
\end{definition}

Let $x_1$ and $x_2$ be two smooth points lying on different components of $X_2$, then
\begin{theorem}
\label{main}
The assignment
\begin{itemize}
\item for all $i=1,2, \Tbi \mapsto \Txi$,
\item $\Ta \mapsto \To$, and
\item $t \mapsto [1]$,
\end{itemize}
defines an action of $\widetilde{\mathrm{PM}}(\bT_2)$ on $D^b(Coh(X_2)).$
\end{theorem}

Note that the assigment described in Theorem \ref{main} is compatible with mirror symmetry considerations, according to which $\cO$ and $\kappa(x_i)$ should be mirror to Lagrangian branes whose supports are isotopic, respectively, to $\alpha$ and $\beta_i$. We state below two lemmas, formulated for a general cycle of projective lines $X_n$, which will be important for proving Theorem \ref{main}. 
\begin{lemma}
\label{map}
Let $F: D^b(Coh(X_n)) \rightarrow D^b(Coh(X_n))$ be an auto-equivalence of triangulated categories. If 
\begin{itemize}
\item $F(\cO) \cong \cO$, and 
\item for all $i \in \{1 \dots n \}$, $F(\gk(x_i)) \cong \gk(x_i)$, 
\end{itemize}
then there exists an isomorphisms $f: X_n \rightarrow X_n$, such that $F$ is naturally equivalent to $f^*: D^b(Coh(X_n)) \rightarrow D^b(Coh(X_n))$.
\end{lemma}
\begin{proof}
See Lemma 3.3 of \cite{Si}. The key fact is that if $F$ preserves $\cO$ and $\gk(x_i)$, then it will induce an automorphism of the homogeneous coordinate ring associated to the ample line bundle $\cL = \cO(x_1 + x_2 + \ldots + x_n)$. This implies the existence of an isomorphism $f:X_n \rightarrow X_n$ such that $F = f^*$, when restricted to the linear sub-category having as objects the tensor powers of $\cL$. The claim then follows as in the proof of Theorem 3.1 of \cite{BO}. Note that, under the assumptions of the theorem, if $n > 2$, $f$ has to be the identity. For $n=2$, $f$ might be non-trivial, but has to be involutive, i.e. $f^2 = id$. 
\end{proof}

\begin{lemma}
\label{comp}
Let $x \in X_n$ be a smooth point, then
\begin{itemize}
\item $T_{\gk(x)} \cong -\otimes \cO(x)$,
\item $\To(\gk(x)) \cong \cO(-x)[1]$,
\item $\To(\cO(x)) \cong \gk(x)$,
\item $\To(\cO) \cong \cO$.
\end{itemize}
\end{lemma}
\begin{proof}
The first isomorphism is proved in \cite{SeT}, Section 3.d. For the other isomorphisms, see Lemma 2.13 in \cite{BK}.
\end{proof}

\begin{proof}[Proof of Theorem \ref{main}]
The braid relations follow from Proposition 2.13 of \cite{SeT}. It remains to check that $\To, \Txi, i=1,2$ satisfy the $\tilde{G}$-relation. Simply by keeping track of the isomorphisms collected in Lemma \ref{comp}, and applying the braid relations, one can see that 
\begin{itemize}
\item $(T_{\gk(x_1)}T_{\cO}T_{\gk(x_2)})^2(\cO) \cong \cO[1]$, and
\item $(T_{\gk(x_1)}T_{\cO}T_{\gk(x_2)})^2(\gk(x_1)) \cong \gk(x_{2})[1]$, $(T_{\gk(x_1)}T_{\cO}T_{\gk(x_2)})^2(\gk(x_2)) \cong \gk(x_{1})[1]$. 
\end{itemize}
Let's check this for $\gk(x_1)$:
$$
(T_{\gk(x_1)}T_{\cO}T_{\gk(x_2)})(T_{\gk(x_1)}T_{\cO}T_{\gk(x_2)})(\gk(x_1)) \cong (T_{\gk(x_1)}T_{\cO}T_{\gk(x_2)})(\cO[1]) \cong \gk(x_2)[1].
$$
Consider an involution $\sigma: X_2 \rightarrow X_2$ such that $\sigma(x_1) = x_2$, and $\sigma(x_2) = x_1$. It follows from Lemma \ref{map}, and the comments made at the end of its proof, that there is an involution $f: X_2 \rightarrow X_2$, and a natural equivalence $(T_{\gk(x_1)}T_{\cO}T_{\gk(x_2)})^2 \cong f^* \sigma^* [1]$. As $\sigma$ and $f$ commute, by taking the square of this natural equivalence, one gets
$$
(T_{\gk(x_1)}T_{\cO}T_{\gk(x_2)})^4 \cong (f^* \sigma^* [1]) (f^* \sigma^* [1]) \cong (f^* )^2 (\sigma^*)^2 [2] \cong [2].
$$
This concludes the proof of the theorem.
\end{proof}

\bibliographystyle{amsplain}

\end{document}